\documentclass{aiml}
 
\usepackage{aimlmacro} 

\usepackage[utf8]{inputenc}
\usepackage{amsmath}
\usepackage{amsfonts}
\usepackage{amssymb}
\usepackage{bussproofs}
\usepackage{tikz}
\usepackage{ stmaryrd }
\usepackage{graphicx}
\usepackage{forest}
\usepackage{xcolor}
\usepackage{mathrsfs}
\usepackage{enumitem}
\usepackage{booktabs}
\usepackage{comment}
\usepackage{cancel}
\usepackage{url}

\usepackage{mathrsfs}

\def\lastname{Eckhardt, Pym}

\begin{document}

\begin{frontmatter}
  \title{Base-extension Semantics for S5 Modal Logic}
  \author{Timo Eckhardt}
  \address{UCL Computer Science\\ London WC1E 6BT, UK, 
  t.eckhardt@ucl.ac.uk}
  \author{David Pym}
  \address{UCL Computer  Science and Philosophy \& Institute of Philosophy \\ London WC1E 6BT, UK, 
  d.pym@ucl.ac.uk}

\begin{abstract}
We develop a proof-theoretic semantics --- in particular, 
a base-extension semantics --- for multi-agent $S5$ modal logic (and hence 
also for the usual unindexed $S5$). Following the inferentialist interpretation of logic, this gives us a semantics in which validity is based on proof rather than truth. 
In base-extension semantics, the validity of formulae is generated by provability in a `base' of atomic rules and an inductive definition of the validity of the connectives. Base-extension semantics for many interesting logics has been explored by several authors and, in particular, a base-extension semantics for the modal logics $K$, $KT$, $K4$, and $S4$ has been developed by the present authors. Here, we give a base-extension semantics for multi-agent $S5$ with $\square_a$, for an agent $a$, as our primary operators, framed as the knowledge operator $K_a$. Similarly to Kripke semantics, we make use of relational structure between bases, allowing us to establish a correspondence between certain bases and worlds. We use this to establish the appropriate soundness and completeness results. We conclude by discussing how this semantics can be extended to Dynamic Epistemic Logics (DEL) starting with Public Announcement Logic (PAL).
\end{abstract}

 \end{frontmatter}

\section{Introduction}

In proof-theoretic semantics (P-tS), meaning is based on inference. It is a 
logical realization of the philosophical position known as inferentialism \cite{brandom2009}. It provides a universal approach to the meaning of the operators of logical systems: their meaning is given in terms of their use. By contrast, in model-theoretic semantics (Mt-S) meaning is determined relative to a choice of abstract (model) structure (see, for example, \cite{Schroeder2007modelvsproof}). 
 
Current research in Pt-S largely follows two different approaches. The first is proof-theoretic validity, following Dummett \cite{Dummett1991} and Prawitz \cite{Prawitz1971,Prawitz2006}. 
It aims to define what makes a proof valid. Dummett's \emph{The Logical Basis of Metaphysics} \cite{Dummett1991} develops a philosophical interpretation of the normality results for the natural deduction rules of intuitionistic propositional logic (IPL). A valid proof is then one that can, by certain fixed operations, be transformed into a proof without unnecessary detours called a canonical proof. This approach is closely related to the Brouwer-Heyting-Kolmogorov interpretation of IPL. For a more in-depth explanation see \cite{SchroederHeister2006,Schroeder2007modelvsproof,SchroederHeister2008}.

The second approach to proof-theoretic semantics is base-extension semantics (B-eS), as developed in, for example, \cite{Makinson2014,piecha2017definitional,Sandqvist2005,Sandqvist2009,Sandqvist2015,SchroederHeister2016}. In B-eS, sets of atomic rules, called bases, establish the validity of atoms and the meaning of the logical connectives is then inductively defined relative to these bases. The choice of the form of atomic rules has profound repercussions: In the usual set-up of B-eS as stated above, limiting the atomic rules to simple production rules yields semantics for classical logic (see \cite{Sandqvist2005,Sandqvist2009}), while allowing for the discharging of atoms as assumptions, and taking care of the treatment of disjunction, results in intuitionistic semantics \cite{Sandqvist2015}. However, see \cite{Goldfarb2016,Stafford2023} for alternative approaches that are related to Kripke semantics. Deep connections between B-eS and the least fixed point semantics for logic programming in IPL, based 
on the proof-theoretic operational semantics of proof-search over hereditary Harrop formulae, such as that described in \cite{Miller89}, have been described in \cite{GP-BSecLog-2023,SchroederHeister2008}. In this setting, definite formulae and negation-as-failure (for intuitionistic propositional logic) are key ideas \cite{GP-BSecLog-2023,SchroederHeister2008}. 

The standard M-tS for modal logic is Kripke semantics \cite{kripke1959,kripke1963}. In this semantics, the validity of formulae is defined by reference to sets of possible worlds and a relation over those worlds. A necessitated formula $\square\phi$ is true at a world if and only if $\phi$ holds at all the worlds reachable via the relation. In multi-agent modal logics, relations are indexed by agents 
and provide models that are able to interpret modalities index by agents. Different modal logics arise according to different restrictions on relations. 

The goal of this paper is to give a Pt-S for multi-agent $S5$ modal logic with $\square_a$, for an agent $a$, as the primary operators. These are framed as the knowledge operator $K_a$, by adapting the Be-S for the modal logics $K$, $KT$, $K4$, and $S4$ given in \cite{Eckhardt2024}. Additionally, we move from single-agent (pure $S5$) to multi-agent $S5$. This simple generalization of $S5$ is achieved by simply considering multiple relations, one for each agent. 
In future work, the multi-agent set-up can be expected to provide the basis for a base-extension semantics for (static) epistemic logic, as used to develop Dynamic Epistemic Logics (DEL) such as Public Announcement Logic (PAL). As such, it seems a worthwhile generalization at this point. 

Similarly to Kripke semantics, relations between bases are used to define validity at a base given a relation. We remark that we do not consider the use of relational structure on bases to be inappropriate from the inferentialist perspective. Indeed, the superset relation is already used for the B-eS of classical propositional logic (and also intuitionistic propositional logic) and we consider that there is no reason why other relations between bases should not be used. Just as in the cases of classical and intutionistic propositional, the validity of modal formulae depends essentially on the atomic rules of those bases reachable via the relation. In the specific set-up of the semantics in \cite{Eckhardt2024}, Euclidean relations and the $5$ axiom turned out not to correspond to each other. We address this issue by changing which relations between bases we consider. 

In Section~\ref{sec:kripke}, we revisit the relational structures required to give Kripke semantics for modal logics.
In Section~\ref{sec:classical-B-eS}, we summarize our recent work, building on Sandqvist's \cite{Sandqvist2005,Sandqvist2009} and Makinson's \cite{Makinson2014} work on treatments of classical propositional logic, on the B-eS for modal logics $K$, $KT$, $K4$, and $S4$. The ideas of these sections will be required for 
Section~\ref{sec:B-eS-S5}, in which we give a B-eS for (multi-agent) 
$S5$. Here, the main challenge is to handle the Euclidean axiom. 
Note that the method we develop here is applicable \emph{only} to $S5$ and --- because some of the resulting proofs rely on the relations between bases being symmetric --- cannot be applied to $K$, $KT$, $K4$, and $S4$. 
In Section~\ref{sec:completeness} we show this semantics to be sound and complete.\footnote{We follow Makinson \cite{Makinson2014}, who gives a base-extension semantics for classical propositional logic on which we base our base-extension semantics for modal logics, in saying that the semantics is sound and complete with respect to a logical system rather than the more usual form that the logical system is sound and complete with respect to the semantics. This also means that the statements of soundness and completeness theorems are swapped compared to how they usually are used.} In Section~\ref{sec:discussion}, we conclude with 
a discussion of how to extend our results to Dynamic Epistemic Logics (DEL), beginning with Public Announcement Logic (PAL). We provide some technical proofs in an appendix. 

\section{Kripke semantics} \label{sec:kripke}


We define a multi-agent Kripke semantics for $S5$ modal logic and give the axioms that constitute sound and complete proof systems. We take $\square$ to be the principal modal operator. As we consider multiple agents, in deference to convention, we use $K_a$ to denote the $\square_a$-operator corresponding to an agent $a$. A single-agent modal logic with only a single $\square$ can be obtained by restricting the set of agents $A$ to a single agent. We also give a sound and complete axiomatic proof system (taken from \cite{Blackburn2001}). 

 A formula $K_a \phi$ holds at a world $w$ iff $\phi$ holds at all worlds $v$ such that $R_a wv$, where $R_a$ is a (reflexive, transitive and Euclidean) relation between worlds. The natural reading for such a formula in an epistemic logic is as `the agent $a$ knows property $\phi$' and, correspondingly, the relation $R_a$ is interpreted as indistinguishability for $a$, that is $R_a wv$ iff at world $w$ agent $a$ cannot distinguish between $w$ and $v$. So, a formula $K_a \phi$ holds at a world $w$ iff at all the worlds that the agent $a$ finds indistinguishable to the world $w$ $\phi$ holds; that is that every world $a$ thinks is possible is a $\phi$ world.

The Kripke semantics and proof system established in this section will be used in Section \ref{sec:completeness} to establish the soundness and completeness of the base-extension semantics of Section \ref{sec:B-eS-S5}.

\begin{definition} 
\label{def:KripkeLanguage}
For atomic propositions $p$ and an agent $a$, the language  for modal logic $S5$ is generated by the following grammar:
\[
    \phi ::= p \mid \bot \mid \neg \phi \mid \phi\to\phi \mid K_a \phi 
\]
\end{definition}

\begin{definition}
    A frame is a pair $F = \langle W, R_A\rangle$, in which $W$ is a set of possible worlds, $R_A$ a set of binary relations on $W$, one for each agent $a\in A$. A model $M= \langle F,V\rangle$ is a pair of a frame $F$ and a valuation function $V$ giving the set of worlds in which $p$ is true s.t. $V(p) \subseteq W$ for every $p$.
\end{definition}

\begin{definition}
 $S5$-frames are those frames whose relations satisfy the frame conditions in Table \ref{BRA} and $S5$-models are those models constructed from $S5$-frames. Frames and models for other modal logics are obtained by only using the corresponding frame conditions (e.g., $KT$-frames are the reflexive frames).
\end{definition}

\begin{definition}
    \label{KripkeValidity}
    Let $F = \langle W, R_A\rangle$ be a $S5$-frame, $M=\langle F,V\rangle$ a $S5$-model and $w\in W$ a world. That a formula $\phi$ is true at $(M,w)$---denoted $M,w\vDash \phi$---as follows:

\[
\begin{array}{l@{\quad}c@{\quad}l}
M,w\vDash p   & \mbox{iff} & w\in V(p) \\
M,w\vDash \phi\to\psi & \mbox{iff} & \mbox{if $M,w\vDash \phi$, then $M,w\vDash \psi$} \\ 
M,w\vDash \neg \phi & \mbox{iff} & \mbox{$M,w \not \vDash \phi$} \\ 
M,w\vDash \bot & \mbox{never} & \\ 
M,w\vDash K_a \phi & \mbox{iff} & \mbox{for all $v$ s.t. $R_a wv$, $M,v\vDash \phi$}\\
\end{array} 
\]

    







\noindent If a formula $\phi$ is true at all worlds in a model $M$, we say $\phi$ is true in $M$. A formula is valid in a modal logic $\gamma$ iff it is true in all $\gamma$-models.


\end{definition}

\begin{table}[t]
\begin{center}
   \[
    \begin{array}{|l@{\quad}|l@{\quad}|l|}
\hline
\mbox{Name} & \mbox{Frame condition} & \mbox{Modal axiom}\\
\hline
K & \mbox{None} & K_a (\phi \to \psi) \to (K_a \phi \to K_a \psi)\\
T & \mbox{Reflexivity} & K_a  \phi \to \phi\\
4 & \mbox{Transitivity} & K_a \phi \to K_a K_a\phi \\
5 & \mbox{Euclidean} & \neg K_a \phi \to K_a\neg K_a\phi \\
\hline
\end{array}
\]
\end{center}
\caption{\label{BRA} Frame conditions and the corresponding modal axioms}
\end{table} 

\begin{definition} \label{HilbertSystem}
    The proof system for the modal logic $S5$, with $a \in A$ is given by the following axioms and rules:
    \begin{itemize}
        \item \emph{(1)}: $\phi\to(\psi\to\phi)$
        \item \emph{(2)}: $(\phi\to(\psi\to\chi))\to((\phi\to\psi)\to(\phi\to\chi))$
        \item \emph{(3)}: $(\neg\phi\to\neg\psi)\to (\psi\to\phi)$
        \item \emph{(K)}: $K_a(\phi\to\psi)\to(K_a\phi\to K_a\psi)$
        \item \emph{(T)}: $K_a \phi\to\phi$
        \item \emph{(4)}: $K_a \phi \to K_a K_a\phi$
        \item \emph{(5)}: $\neg K_a \phi\to K_a\neg K_a\phi$
        \item \emph{(MP)}: From $\phi$ and $\phi\to\psi$, infer $\psi$ 
        \item \emph{(NEC)}: From $\phi$, infer $K_a \phi$.
    \end{itemize}
    \noindent The axioms \emph{(1)-(3)} together with the rule \emph{(MP)} constitute a proof system for classical logic. Proof systems for the modal logics other than $S5$ are obtained by taking only the corresponding axioms from Table \ref{BRA} (e.g., modal logic $KT$ results from adding \emph{(K)}, \emph{(T)}, and \emph{(NEC)} to the classical proof system).
\end{definition}

Given the Kripke semantics and the corresponding axiomatic proof system we can state the following, well-known, result.

\begin{theorem}
\label{KripkeComp}
The proof system for a multi-agent modal logic $S5$ is sound and complete with respect to the validity given by $S5$-models. 
\end{theorem}

\section{Preliminary base-extension semantics} \label{sec:classical-B-eS}


In this section, we define the base-extension semantics for classical logic that will form the foundation for our modal semantics. Our presentation of classical base-extension semantics follows the versions given in \cite{Sandqvist2005,Sandqvist2009} and \cite{Makinson2014}. We also briefly introduce the modal base-extension semantics for the single-agent modal logics $K$, $KT$, $K4$, and $S4$ we developed in \cite{Eckhardt2024}. 

We define base rules and bases to give the atomic system that will underlie the definition of validity. Intuitively, a base rule is an inference rules for atomic formulae and a base is a set of such rules. This gives us a deducibility relation at a base for atomic formulae. 
This relation is first extended with semantic clauses to give a full consequence relation for the classical connectives: a formula is taken to be valid if it holds at all bases. The relation is 
then similarly further extended to encompass the modal operators of $K$, $KT$, $K4$, and $S4$, as developed in \cite{Eckhardt2024}. 

We also define the notion of maximally-consistent bases (i.e., bases containing a maximum amount of base rules without becoming inconsistent). For the single-agent modal base-extension semantics for non-Euclidean systems, we add a relation between bases and define the appropriate relations as necessary, called modal relations. We then define validity based on holding at a base given a modal relation. The notion of modal relation is especially important as, in \cite{Eckhardt2024}, we have shown that it results in incomplete semantics for Euclidean modal logics and so will have to be changed for the $S5$ modal logic in Section \ref{sec:B-eS-S5}.


\subsection{Classical base-extension semantics}

We start our presentation of classical base-extension semantics by defining the language.

\begin{definition} 
\label{def:classicalLanguage}
For atomic propositions $p$, the language for classical logic is generated by the following grammar:
\[
\phi ::= p \mid \bot \mid \phi \to \phi 
\]
\end{definition}

The negation symbol $\neg \phi$ is defined as $\phi\to\bot$. All other connectives can be obtained from $\to$ and $\bot$ in the usual manner.

As base rules and bases build the foundation of validity, we have to define them before moving to validity conditions.

\begin{definition} \label{ClassicalRule}
A \textit{base rule} is a pair $(L_j, p)$ where $L_j  = \{p_1, \dots, p_n \}$ is a finite (possibly empty) set of atomic formulae and $p$ is also an atomic formula. Generally, a base rule will be written as $p_1, \dots, p_n \Rightarrow p$ ($\Rightarrow p$ for axioms). A \textit{base} $\mathscr{B}$ is any countable collection of base rules. We call the set of all bases $\Omega$ and 
$\overline{\mathscr{B}}$ is the closure of the empty set under the rules in $\mathscr{B}$.
\end{definition}

An atomic formula $p$ is provable in a base $\mathscr{B}$ 
if $p$ can be derived using the rules in $\mathscr{B}$. 

Given these definitions, we can now define the validity of formulae at a base. The validity of atomic formulae at a base is defined as provability in a base: this is the essential foundation of base-extension semantics. The conditions for the validity of complex formulae at a base are then inductively defined in a familiar way. The validity of a formula is then defined by that formula being valid at all bases.

\begin{definition} \label{satisfaction}
Classical validity is defined as follows:
\[
\begin{array}{l@{\quad}c@{\quad}l}
 \Vdash_\mathscr{B} p    & \mbox{iff} & \mbox{$p$ is in 
            every set of atomic formulae closed under $\mathscr{B}$} \\  
        & & \mbox{(i.e., iff $p\in \overline{\mathscr{B}}$)} \\
\Vdash_\mathscr{B} \phi\to\psi & \mbox{iff} & \mbox{$\phi\Vdash_\mathscr{B} \psi$} \\ 
\Vdash_\mathscr{B} \bot & \mbox{iff} & \mbox{$\Vdash_{\mathscr{B}} p$ for every atomic formula $p$.} \\
& & \\
 \mbox{For non empty $\Gamma$:} & &\\
\Gamma\Vdash_{\mathscr{B}} \phi & \mbox{iff} & \mbox{for all $\mathscr{C}\supseteq \mathscr{B}$, if $\Vdash_{\mathscr{C}} \psi$ for all $\psi \in \Gamma$, then $\Vdash_{\mathscr{C}} \phi$}\\ 
\end{array} 
\]
A formula $\phi$ is \textit{valid} iff $\Vdash_\mathscr{B} \phi$ for every base $\mathscr{B}$.
A base $\mathscr{B}$ is {\it inconsistent} iff $\Vdash_\mathscr{B} \bot$ and {\it consistent} otherwise.
\end{definition}

For classical base-extension semantics \cite{Makinson2014} has shown that maximally-consistent bases correspond to valuations in classical semantics. In \cite{Eckhardt2024}, we have shown that the same holds true between maximally-consistent bases in modal base-extension semantics and worlds in Kripke semantics. For that reason we define maximally-consistent bases here.

\begin{definition}
    \label{MaxConDef}
    A base $\mathscr{B}$ is {\it maximally-consistent} iff it is consistent and for every base rule $\delta$, either $\delta \in \mathscr{B}$ or $\mathscr{B}\cup\{\delta\}$ is inconsistent.

\end{definition}

Further we show that if a formula does not hold at some base, there also has to be a maximally-consistent base at which it is not valid. This will prove useful in our proof of soundness later.

\begin{lemma}
\label{MaxCon}
For every propositional formula $\phi$, if there is a base $\mathscr{B}$ s.t. $\nVdash_\mathscr{B} \phi$, then there is a maximally-consistent base $\mathscr{B}^*\supseteq\mathscr{B}$ with $\nVdash_{\mathscr{B}^*} \phi$.
\end{lemma}

\noindent A proof of Lemma \ref{MaxCon} can be found in \cite{Makinson2014}.

\subsection{Single-agent non-Euclidean modal logic} \label{subsec:modal}

For our modal language we extend the language for classical base-extension semantics with a $\square$ operator.

\begin{definition}
For atomic propositions $p$, the language for a modal logic $\gamma$ is generated by the following grammar:
\[
\begin{array}{rcl}
	\phi & := & p \mid \bot \mid \phi \rightarrow \phi \mid \square \phi 
\end{array}
\]
\end{definition}

Importantly, we do not change the make-up of our base rules and bases and so the definitions of those terms from Definition \ref{ClassicalRule} remain unchanged. 

The fundamental difference between the classical and modal accounts lies in the definition of validity. We no longer define validity at a base, but rather validity at a base given a relation over the set of bases. So, before moving to validity, we must define exactly the type of relations between bases we require.

We stress at this point that the need in the formulating base-extension semantics for modal logics for relational structures that are reminiscent of the relational structures taken in Kripke semantics does not imply that base-extension semantics is a form of Kripke semantics. This situation is similar to that encountered in 
the base-extension semantics on intutionistic propositional logic 
(e.g., \cite{Sandqvist2015}), where a superset relation on bases is required. 
In both settings, the essential content of base-extension semantics is its 
grounding in the provability of atoms in bases rather than the valuation of atoms in model-theoretic structures.  

\begin{definition} \label{OldmodalRelation}
A relation $\mathfrak{R}$ on the set of bases $\Omega$ is called a modal relation iff, for all $\mathscr{B}$,

\begin{enumerate}[label=(\alph*)]
    \item if $\Vdash_{\mathscr{B}} \bot$, then there is a $\mathscr{C}$ s.t. $\mathfrak{R}\mathscr{B}\mathscr{C}$ and $\Vdash_{\mathscr{C}} \bot$ and, for all $\mathscr{D}$, if $\mathfrak{R}\mathscr{B}\mathscr{D}$, then $\Vdash_{\mathscr{D}} \bot$ 
    \item if $\nVdash_{\mathscr{B}} \bot$, then, for all $\mathscr{C}$, s.t. $\mathfrak{R} \mathscr{B}\mathscr{C}$, $\nVdash_{\mathscr{C}} \bot$
    \item for all $\mathscr{C}$, if $\mathscr{B}$ is consistent and $\mathfrak{R}\mathscr{B}\mathscr{C}$, then either $\mathscr{B}$ is maximally-consistent or there is a $\mathscr{D}\supset \mathscr{B}$ s.t. $\mathfrak{R}\mathscr{D}\mathscr{C}$
    \item for all $\mathscr{C}$, if $\mathfrak{R}\mathscr{B}\mathscr{C}$, then, for all $\mathscr{D}\subseteq\mathscr{B}$, $\mathfrak{R}\mathscr{D}\mathscr{C}$.
\end{enumerate}

\noindent Given a modal logic $\gamma = K, K4, KT$, or $S4$, a modal relation $\mathfrak{R}_a$ is called a $\gamma$-modal relation iff $\mathfrak{R}_a$ satisfies the frame conditions corresponding to $\gamma$.
\end{definition}

Conditions (a) and (b) make sure that inconsistent and consistent worlds are not connected via $\mathfrak{R}$ and conditions (c) and (d) guarantee that the structure of $\mathfrak{R}$ is preserved when going from a base to its super- and sub-sets, respectively. For a more in-depth discussion of this definition see \cite{Eckhardt2024}.

We can now give a general definition of validity for the modal logics $K, KT, K4$ and $S4$, as they share the same validity conditions and only differ in the modal relations they consider.

\begin{definition}
 \label{OldEXTValidity}

We define validity at a base $\mathscr{B}$ given a $\gamma$-modal relation $\mathfrak{R}$ for a modal logic $\gamma$ as follows:
\[
\begin{array}{l@{\quad}c@{\quad}l}
\Vdash^\gamma_{\mathscr{B},\mathfrak{R}} p   & \mbox{iff} & \mbox{$p$ is in every set of atomic formulae closed under $\mathscr{B}$} \\  
        & & \mbox{(i.e., iff $p\in \overline{\mathscr{B}}$)} \\
\Vdash^\gamma_{\mathscr{B},\mathfrak{R}} \phi\to\psi & \mbox{iff} & \mbox{$\phi\Vdash^\gamma_{\mathscr{B},\mathfrak{R}} \psi$} \\ 
\Vdash^\gamma_{\mathscr{B},\mathfrak{R}} \bot & \mbox{iff} & \mbox{$\Vdash^\gamma_{\mathscr{B},\mathfrak{R}} p$ for every atomic formula $p$} \\
\Vdash^\gamma_{\mathscr{B},\mathfrak{R}} \square \phi & \mbox{iff} & \mbox{for all $\mathscr{C}\supseteq \mathscr{B}$ and $\mathscr{C'}$ s.t. $\mathfrak{R}\mathscr{C}\mathscr{C'}$, $\Vdash^\gamma_{\mathscr{C'},\mathfrak{R}} \phi$}\\
& & \\
 \mbox{For non empty $\Gamma$:} & &\\
\Gamma\Vdash^\gamma_{\mathscr{B},\mathfrak{R}_A} \phi & \mbox{iff} & \mbox{for all $\mathscr{C}\supseteq \mathscr{B}$, if $\Vdash^\gamma_{\mathscr{C},\mathfrak{R}_A} \psi$ for all $\psi \in \Gamma$, then $\Vdash^\gamma_{\mathscr{C},\mathfrak{R}_A} \phi$}\\ 

\end{array} 
\]








\noindent A formula $\phi$ is \textit{$\gamma$-valid}, written as $\Vdash^\gamma \phi$, iff $\Vdash^\gamma_{\mathscr{B}, \mathfrak{R}} \phi$ for all modal bases $\mathscr{B}$ and $\gamma$-modal relations $\mathfrak{R}$. 
 A base $\mathscr{B}$ is {\it inconsistent} in $\gamma$ iff $\Vdash^\gamma_{\mathscr{B}, \mathfrak{R}_A} \bot$ for any $\gamma$-modal relation $\mathfrak{R}_A$.


\end{definition}

\section{Base-extension semantics for $S5$} \label{sec:B-eS-S5}

Our presentation of modal base-extension semantics for the $S5$ multi-agent modal logic only differs in a few key places from the base-extension semantics for modal logics given above. In our presentation, perhaps the most obvious difference is that we use, in deference to convention,  $K_a$ as the knowledge operator for an agent $a$ to replace $\square$. 
As we want to consider the knowledge of multiple agents, we add a modal operator $K_a$ for every agent $a$ in a given set of agents $A$, as we did for Kripke semantics. Correspondingly, we consider relations between bases $\mathfrak{R}_a$ for each agent $a$. Again this relation can be interpreted as indistinguishability so that $\mathfrak{R}_a \mathscr{B}\mathscr{C}$ iff agent $a$ cannot distinguish between bases $\mathscr{B}\mathscr{C}$.


For single agent $a$, so that the relation is just an unindexed $R$, this modality reduces to the usual S5 necessitation, $\square$. It follows that our analysis incorporates this case.  

The remainder of this section develops a base-extension semantics for multi-agent $S5$ by adapting the definitions from Section \ref{sec:classical-B-eS}. We begin by setting up  validity and give a definition of modal relation capable of handling $S5$. We proceed to establish completeness and soundness in Section \ref{sec:completeness}. 

\begin{definition}
    For atomic propositions $p$ and agents $a$, the language for the epistemic modal logic is generated by the following grammar:

    \[\phi := p \mid \bot \mid \phi\to\phi \mid K_a \phi\]
\end{definition}

Importantly, the definitions for base rules and bases remain unchanged from the classical semantics and the modal semantics presented above as given in Definition \ref{ClassicalRule}. We now no longer need to simply consider a single modal relation $\mathfrak{R}$ but rather a $\mathfrak{R}_a$ for every agent $a \in A$. So, we start by defining validity at a base given a set of relations and give validity (proper) based on this notion. We write $\Vdash$ instead of $\Vdash^{S5}$ as we do not consider any other modal logic in the rest of this paper.

\begin{definition}
 \label{EXTValidity}

For the epistemic modal logic, we define validity at a base $\mathscr{B}$ given a set of agents $A$ and a set of $S5$-modal relations $\mathfrak{R}_a$, one for every agent $a \in A$, denoted by $\mathfrak{R}_A$ as follows:
\[
\begin{array}{l@{\quad}c@{\quad}l}
\Vdash_{\mathscr{B},\mathfrak{R}_A} p   & \mbox{iff} & \mbox{$p$ is in every set of atomic formulae closed under $\mathscr{B}$} \\  
        & & \mbox{(i.e., iff $p\in \overline{\mathscr{B}}$)} \\
       
\Vdash_{\mathscr{B},\mathfrak{R}_A} \phi\to\psi & \mbox{iff} & \mbox{$\phi\Vdash_{\mathscr{B},\mathfrak{R}_A} \psi$} \\ 
\Vdash_{\mathscr{B},\mathfrak{R}_A} \bot & \mbox{iff} & \mbox{$\Vdash_{\mathscr{B},\mathfrak{R}_A} p$ for every atomic formula $p$} \\
\Vdash_{\mathscr{B},\mathfrak{R}_A} K_a \phi & \mbox{iff} & \mbox{for all $\mathscr{C}$ s.t. $\mathfrak{R}_a\mathscr{B}\mathscr{C}$, $\Vdash_{\mathscr{C},\mathfrak{R}_A} \phi$}\\
& & \\
 \mbox{For non empty $\Gamma$:} & &\\
\Gamma\Vdash_{\mathscr{B},\mathfrak{R}_A} \phi & \mbox{iff} & \mbox{for all $\mathscr{C}\supseteq \mathscr{B}$, if $\Vdash_{\mathscr{C},\mathfrak{R}_A} \psi$ for all $\psi \in \Gamma$, then $\Vdash_{\mathscr{C},\mathfrak{R}_A} \phi$}\\ 
\end{array} 
\]








\noindent A formula $\phi$ is \textit{valid}, written as $\Vdash \phi$, iff $\Vdash_{\mathscr{B}, \mathfrak{R}_A} \phi$ for all modal bases $\mathscr{B}$ and set of modal relations $\mathfrak{R}_A$. A base $\mathscr{B}$ is {\it inconsistent} (in $S5$) iff $\Vdash_{\mathscr{B},\mathfrak{R}_A} \bot$ for any $\mathfrak{R}_A$.


\end{definition}

As one might expect, the definition of validity for the classical connectives remains unchanged. The new case for $K_a \phi$ is defined in much the same way as it is in Kripke semantics: by universal quantification over the bases that are in relation to the base we started at. There are, however, two differences in our presentation. As mentioned above, we have a $K_a$ operator for every agents $a \in A$ and so have different $\mathfrak{R}_a$ to consider. Additionally, in Definition \ref{OldEXTValidity} we required the validity condition for $\square$ as it is here to hold for all supersets of $\mathscr{B}$ but, as follows, we also change the definition of modal relation from Definition \ref{OldmodalRelation} and it is no longer required. 

This is due to the way we changed the conditions for modal relations from Definition \ref{OldmodalRelation}. We give the new conditions below in Definition \ref{modalRelation}. We can not just allow for any relation on the set of bases to be a modal relation $\mathfrak{R}$. 

First, as in Kripke semantics, putting restrictions in form of the well-known frame conditions of modal logic on $\mathfrak{R}$ will allow us to model different modal logics. 

Here we want to give a base-extension semantics for $S5$ modal logic as our starting off point for PAL. As might be expected, requiring $\mathfrak{R}$ to be reflexive, transitive, and Euclidean will be necessary to get a base-extension semantics for $S5$ modal logic. We show this below as part of our completeness proof.

There are, however, a couple more restrictions that need to be put on $\mathfrak{R}$, regardless of the specific modal logic. In comparison to Kripke models, there is more inherent structure embedded into base-extension semantics. We now have two types of relations between bases: $\subseteq$ and $\mathfrak{R}_a$ and we have inconsistent bases to deal with. We impose conditions to ensure that these behave appropriately. 

For this we update the conditions (a)-(d). Conditions (c) and (d) are the big difference from the presentation in Definition \ref{OldmodalRelation}. 

In Definition \ref{OldmodalRelation}, the ideas behind (c) and (d) are that if a base $\mathscr{C}$ is considered possible at a base $\mathscr{B}$ (i.e., is reachable via a modal relation $\mathfrak{R}$), it has to have been reachable from all the subsets and at least one of the supersets of $\mathscr{B}$. This captures that adding rules to a base reduces the number of 
bases reachable from it. 

This view, however, causes the $5$ axiom and Euclidean to no longer correspond to each other, as we have shown in \cite{Eckhardt2024}. 

In fact, given an equivalence relation such a relation will collapse to a universal relation between bases. At every base $\mathscr{D}$, we have $\mathfrak{R}_a\mathscr{D}\mathscr{D}$ and $\mathfrak{R}_a\emptyset\emptyset$, by reflexivity. Note that $\mathscr{D}\supseteq\emptyset$ and so, by condition (c), we also have $\mathfrak{R}_a\emptyset\mathscr{D}$ and, by symmetry, $\mathfrak{R}_a\mathscr{D}\emptyset$. So for any two bases $\mathscr{D}_1$ and $\mathscr{D}_2$ we get $\mathfrak{R}_a\mathscr{D}_1\emptyset$ and $\mathfrak{R}_a\emptyset\mathscr{D}_2$ and, finally, $\mathfrak{R}_a\mathscr{D}_1\mathscr{D}_2$ by transitivity.

Intuitively, we require conditions (c) and (d) that better preserve the structure of a relation between sub- and super-set bases.

We present the following new conditions for a modal relation to address this issue.


\begin{definition}
\label{modalRelation}
A relation $\mathfrak{R}_a$ on the set of bases $\Omega$ is called a  $S5$-modal relation iff, for all $\mathscr{B}$,


\begin{enumerate}[label=(\alph*)]
    \item if $\Vdash_{\mathscr{B}} \bot$, there is a $\mathscr{C}$ s.t. $\mathfrak{R}_a\mathscr{B}\mathscr{C}$ and $\Vdash_{\mathscr{C}} \bot$ and, for all $\mathscr{D}$, if $\mathfrak{R}_a\mathscr{B}\mathscr{D}$, then $\Vdash_{\mathscr{D}} \bot$ 

    \item if $\nVdash_{\mathscr{B}} \bot$, then for all $\mathscr{C}$, s.t. $\mathfrak{R}_a \mathscr{B}\mathscr{C}$, $\nVdash_{\mathscr{C}} \bot$

    \item for all $\mathscr{C}$, if $\mathfrak{R}_a\mathscr{B}\mathscr{C}$, then, for all consistent $\mathscr{D}\supseteq \mathscr{B}$, there is an $\mathscr{E}\supseteq\mathscr{C}$ s.t. $\mathfrak{R}_a\mathscr{D}\mathscr{E}$

    \item for all consistent $\mathscr{C}$, if $\mathfrak{R}_a\mathscr{B}\mathscr{C}$, then, for all $\mathscr{D}\subseteq\mathscr{B}$, there are $\mathscr{E}\subseteq\mathscr{C}$ s.t. $\mathfrak{R}_a\mathscr{D}\mathscr{E}$.


\end{enumerate}

\noindent A modal relation $\mathfrak{R}$ is called a $S5$-modal relation iff $\mathfrak{R}$ is reflexive, transitive, and Euclidean.

\end{definition}

Conditions (a) and (b) govern inconsistent bases by making sure that inconsistent bases all have access to at least one other inconsistent base and no consistent bases (thereby forcing $K_a \bot$ and $\neg K_a\neg\bot$ at them) and that no consistent base can have access to an inconsistent one, much like no world in Kripke semantics can have access to an inconsistent world(see Figure \ref{11(ab)}).

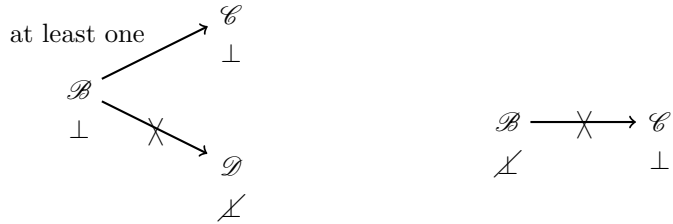
\begin{figure}[ht]
\begin{center}
\begin{tabular}{cc}

\begin{tikzpicture}[-,shorten >=1pt,node distance =1cm, thick] 
			\node[label=below:{$\bot$}]  (A) at (0,1) {$\mathscr{B}$};
			\node[label=below:{$\bot$}] (C) at (2,2) {$\mathscr{C}$};
			\node[label=below:{$\cancel{\bot}$}] (D) at (2,0) {$\mathscr{D}$};
			\path

(A) edge[->] node[above] {at least one\qquad\:\;\;\:\;\;\;\;\:\;\:\:\;\;\;     } (C)
;		
\draw[->] (A) to node {$\xcancel{\:\;}$} (D);

		\end{tikzpicture}

&

\begin{tikzpicture}[-,shorten >=1pt,node distance =2cm, thick] 

\node[label=below:{$\;$}] (Z) at (0,0) {$\;$};
\node[label=below:{$\cancel{\bot}$}] (A) at (2,1) {$\mathscr{B}$};
\node[label=below:{$\bot$}] (B) [right of=A] {$\mathscr{C}$};

\draw[->] (A) to node {$\xcancel{\:\;}$} (B);

\end{tikzpicture}

\end{tabular}
\end{center}
\caption{\label{11(ab)} Illustration of Definition \ref{modalRelation} (a) on the left and (b) on the right}
\end{figure}

Additionally, we need conditions that ensure that the structure of $\mathfrak{R}_a$ is preserved to super- and subset bases. The conditions (c) and (d), as they are given here, directly correspond to the \emph{zig and zag} conditions for bisimulation (see, for example \cite{Blackburn2001}). This makes them obvious candidates for restrictions to put on modal relations as zig and zag similarly guarantee that the structure of a relation is preserved when going back or forth on the bisimulation. 

Given the reading of our logic as an epistemic logic, there is another interesting observation about (c) and (d). In game theory~\cite{Kuhn1953} and dynamic logics~\cite{vanBenthem2001} there are two related notions that formally coincide with zig and zag: \emph{perfect recall} and \emph{no learning}. Perfect recall expresses the idea that no information is lost from previous states. If we interpret a superset base as one that has more information in it, because it has more rules in it, (d) ensures that no bases are considered that were not already considered before and so, that nothing that was previously known becomes lost by adding those new rules. No learning is the dual of perfect recall and expresses the idea that no information about the possibility of bases is gained by adding these rules and so directly corresponds to (c). Perfect recall and \emph{no miracles}, a weaker version of no learning that expresses that no learning of unrelated information can occur, are inherent to updates in PAL and DEL~\cite{Wang2013,Eckhardt2021}. We return to a brief discussion of further work on PAL in Section~\ref{sec:discussion}.

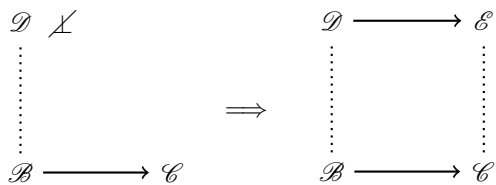
\begin{figure}[ht]
\begin{center}

\begin{tabular}{c p{.7cm} c}

\begin{tikzpicture}[-,shorten >=1pt,node distance =2cm, thick, baseline={([yshift=-1.8ex]current bounding box.center)}] 

\node (A) {$\mathscr{B}$};
\node (B) [right of=A] {$\mathscr{C}$};
\node[label=right:{$\cancel{\bot}$}] (C) [above of=A] {$\mathscr{D}$};

\draw[->] (A) to node {} (B);
\draw[dotted, right] (A) to node {} (C);

\end{tikzpicture}

&

$\Longrightarrow$

&

\begin{tikzpicture}[-,shorten >=1pt,node distance =2cm, thick, baseline={([yshift=-1.8ex]current bounding box.center)}] 

\node (A) {$\mathscr{B}$};
\node (B) [right of=A] {$\mathscr{C}$};
\node (C) [above of=A] {$\mathscr{D}$};
\node (D) [right of=C] {$\mathscr{E}$};

\draw[->] (A) to node {} (B);
\draw[dotted, right] (A) to node {} (C);
\draw[->, right] (C) to node {} (D);
\draw[dotted, right] (B) to node {} (D);

\end{tikzpicture}

\end{tabular}
\end{center}
\caption{\label{11(c)} Illustration of Definition \ref{modalRelation} (c) with dotted lines representing the subset relation}

\end{figure}

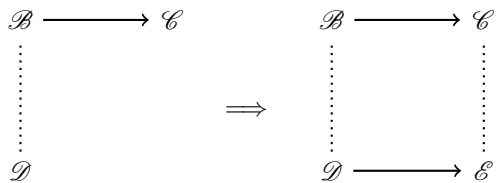
\begin{figure}[ht]
\begin{center}
\begin{tabular}{c p{.7cm} c}

\begin{tikzpicture}[-,shorten >=1pt,node distance =2cm, thick, baseline={([yshift=-1.8ex]current bounding box.center)}] 

\node (A) {$\mathscr{D}$};
\node (C) [above of=A] {$\mathscr{B}$};
\node (D) [right of=C] {$\mathscr{C}$};

\draw[->] (C) to node {} (D);
\draw[dotted, right] (A) to node {} (C);

\end{tikzpicture}

& 

$\Longrightarrow$

&

\begin{tikzpicture}[-,shorten >=1pt,node distance =2cm, thick, baseline={([yshift=-1.8ex]current bounding box.center)}] 

\node (A) {$\mathscr{D}$};
\node (B) [right of=A] {$\mathscr{E}$};
\node (C) [above of=A] {$\mathscr{B}$};
\node (D) [right of=C] {$\mathscr{C}$};

\draw[->] (A) to node {} (B);
\draw[dotted, right] (A) to node {} (C);
\draw[->, right] (C) to node {} (D);
\draw[dotted, right] (B) to node {} (D);

\end{tikzpicture}

\end{tabular}

\caption{\label{11(d)} Illustration of Definition \ref{modalRelation} (d) with dotted lines representing the subset relation}
\end{center}
\end{figure}

\iffalse{
\begin{definition} \label{def:classical-nd}
The natural deduction system for classical propositional logic is given by the following rules:

\begin{center}
\AxiomC{$[\phi]$}
\noLine
\UnaryInfC{$\psi$}
\RightLabel{$(\to I)$}
\UnaryInfC{$\phi\to\psi$}
\DisplayProof
\quad
\AxiomC{$\phi\to \psi$}
\AxiomC{$\phi$}
\RightLabel{$(\to E)$}
\BinaryInfC{$\psi$}
\DisplayProof
\end{center}

\begin{prooftree}
\AxiomC{$[\phi\to\bot]$}
\noLine
\UnaryInfC{$\bot$}
\RightLabel{$(\bot_c)$}
\UnaryInfC{$\phi$}
\end{prooftree}

If we just have the rules $(\to I)$ and $(\to E)$, we have a natural deduction system for minimal propositional logic.
  
\end{definition}
}\else

We can now show that some of the most important properties of base-extension semantics can be generalized to our new system. The end goal is to show that our semantics is in fact sound and complete with respect to different Euclidean modal logics. The general idea for this will, again, be close to what was shown for the non-Euclidean base-extension semantics in \cite{Eckhardt2024} but given our changed definition of modal relation we have to make sure that all steps of these proofs still hold. Additionally, as the axiomatic systems for modal logic are easier to work with, we showed that the modal axioms are valid but followed \cite{Sandqvist2005,Sandqvist2009} in showing the classical fragment with respect to a natural deduction proof system in \cite{Eckhardt2024}. Here we show completeness with respect to the axiomatic systems 
for modal logic rather than the natural deduction system. 

As mentioned above, we no longer need to consider superset bases for the validity of formulae of the form $K_a\phi$ (as in \cite{Eckhardt2024}). This is because the condition (d) guarantees that if $K_a \phi$ holds at a base, then it must also hold at its supersets. We show this as part of Lemma \ref{ModalMonotonicity}. The appeal to superset bases in the validity condition for implications, say $\phi\to\psi$, is essentially a hypothetical one: we check all extensions in which $\phi$ hold and see whether $\psi$ holds at those as well. There is no such hypothetical reasoning required for $K_a\phi$. We simply care about the bases reachable by agent $a$ from our initial base. This suggests that the conditions here are more suited for a modal base-extension semantics as the ones used in \cite{Eckhardt2024}. 


\begin{lemma}
    \label{ModalMonotonicity}

If $\Gamma\Vdash_{\mathscr{B},\mathfrak{R}_A} \phi$ and $\mathscr{B}\subseteq\mathscr{C}$, then $\Gamma\Vdash_{\mathscr{C},\mathfrak{R}_A} \phi$.
    
\end{lemma}

\begin{proof}
    This proof goes by induction on the complexity of $\phi$. If $\Gamma$ is not empty, this follows immediately from the validity conditions. 
    
    For empty $\Gamma$, the propositional cases remain unchanged from the proof for classical propositional logic in \cite{Sandqvist2009}. The new case is the case in which $\phi=K_a\psi$. In order for $\Vdash_{\mathscr{B},\mathfrak{R}_A} K_a\psi$, at all $\mathscr{D}$ s.t. $\mathfrak{R}_a\mathscr{B}\mathscr{D}$, $\Vdash_{\mathscr{D},\mathfrak{R}_A} \psi$.
    By condition (d) of Definition \ref{modalRelation}, we know that for every $\mathscr{E}$ s.t. $\mathfrak{R}_a\mathscr{C}\mathscr{E}$, there is such a $\mathscr{D}$ with $\mathscr{D}\subseteq\mathscr{E}$. By induction hypothesis $\Vdash_{\mathscr{E},\mathfrak{R}_A} \psi$ and, since this holds for all $\mathscr{E}$ s.t. $\mathfrak{R}_a\mathscr{C}\mathscr{E}$, we conclude $\Vdash_{\mathscr{C},\mathfrak{R}_A} K_a\psi$.
\end{proof}

Additionally, the conditions of Definition \ref{modalRelation} guarantee {\it ex falso quod libet} in the sense that any formula $\phi$ follows from $\bot$.

\begin{lemma}
    \label{EFQ}
    For all inconsistent bases $\mathscr{B}$ and $\mathfrak{R}_A$, $\Vdash_{\mathscr{B},\mathfrak{R}_A}\phi$ for all $\phi$.
    
\end{lemma}

\begin{proof}
    This proof goes by induction on the complexity of the formula $\phi$. For propositional $\phi$ this follows immediately from the validity conditions in Definition \ref{EXTValidity}. For the case in which $\phi=K_a\psi$, the condition (a) ensures that any inconsistent base can only have access to other inconsistent bases and, by induction hypothesis, $\psi$ holds at these bases.
\end{proof}

In Definition \ref{MaxConDef}, we defined what it means for a base to be maximally-consistent. In \cite{Eckhardt2024} we used maximally-consistent bases to prove soundness for a variety of modal base-extension semantics and we follow this strategy again here. This is because, if we carefully pick $\mathfrak{R}_A$, maximally-consistent bases closely correspond to worlds in Kripke semantics. 

In \cite{Makinson2014}, Makinson shows that for classical propositional base-extension semantics, formulae on maximally-consistent bases can be evaluated using their classical valuations. In Lemma \ref{ModalBehaviour}, this result is generalized to show that formulae on maximally-consistent bases can be evaluated using validity conditions that directly correspond to the truth conditions  in Kripke semantics in our modal base-extension semantics. 

\begin{lemma}
\label{ModalBehaviour}
For any set of $S5$-modal relations $\mathfrak{R}_A$ and maximally-consistent base $\mathscr{B}$, the following hold:

\begin{itemize}
    \item $\nVdash_{\mathscr{B},\mathfrak{R}_A} \bot$, and
    \item $ \Vdash_{\mathscr{B},\mathfrak{R}_A} \phi \to \psi$ iff $\nVdash_{\mathscr{B},\mathfrak{R}_A} \phi$ or $\Vdash_{\mathscr{B},\mathfrak{R}_A} \psi$.
\end{itemize}

\end{lemma}

\begin{proof}
    As the validity conditions for $\to$ and $\bot$ do not depend on $\mathfrak{R}_A$  the proof remains unchanged from \cite{Makinson2014}.
\end{proof}

From this it follows, maybe unsurprisingly, that the {\it law of excluded middle} holds at maximally-consistent bases.

\begin{lemma}
    \label{lem:excludedmiddle}
    For every formula $\phi$, maximally-consistent base $\mathscr{B}$, and set of $S5$-modal relation $\mathfrak{R}_A$, either $\Vdash_{\mathscr{B},\mathfrak{R}_A} \phi$ or $\Vdash_{\mathscr{B},\mathfrak{R}_A} \phi\to\bot$.
\end{lemma}

\begin{proof}
    Consider the case in which $\nVdash_{\mathscr{B},\mathfrak{R}_A} \phi$. Any base $\mathscr{C}$ s.t. $\mathscr{C}\supseteq\mathscr{B}$ and $\Vdash_{\mathscr{C},\mathfrak{R}_A} \phi$ is inconsistent and so $\Vdash_{\mathscr{C},\mathfrak{R}_A} \bot$ and so $\Vdash_{\mathscr{B},\mathfrak{R}_A} \phi\to\bot$.

\end{proof}

Before moving to our completeness and soundness proofs, we show some more useful properties of maximally-consistent bases. If a formula is not valid at a particular base, then there is a maximally-consistent superset of that base at which it does not hold either. Or conversely, if something holds at all maximally-consistent supersets of a base, it also has to hold at that base. 

\begin{lemma}
    \label{lem:nonvalidMaxcon}
    For every formula $\phi$, if there is a base $\mathscr{B}$ s.t. $\nVdash_{\mathscr{B},\mathfrak{R}_A} \phi$, then there is a maximally-consistent base $\mathscr{C}\supseteq \mathscr{B}$ with $\nVdash_{\mathscr{C},\mathfrak{R}_A} \phi$.
\end{lemma}

\begin{proof}
    The proof goes by induction on the complexity of the formula $\phi$. The proof for the classical cases follows the proof in \cite{Makinson2014}.

    We consider the case in which $\phi=K_a\psi$. Since $\nVdash_{\mathscr{B},\mathfrak{R}_A} K_a\psi$, there is a $\mathscr{C}$ s.t. $\mathfrak{R}_a\mathscr{B}\mathscr{C}$ and $\nVdash_{\mathscr{C},\mathfrak{R}_A} \psi$. By our induction hypothesis, we know that there is a maximally-consistent $\mathscr{C}^*\supseteq\mathscr{C}$ with $\nVdash_{\mathscr{C},\mathfrak{R}_A} \psi$. By symmetry, $\mathfrak{R}_a\mathscr{C}\mathscr{B}$ and so, by condition (c) of Defintition \ref{modalRelation}, there is a $\mathscr{B}'\supseteq\mathscr{B}$ s.t. $\mathfrak{R}_a\mathscr{C}^*\mathscr{B}'$. Again by symmetry, we have $\mathfrak{R}_a\mathscr{B}'\mathscr{C}^*$. If $\mathscr{B}'$ is maximally-consistent, let $\mathscr{B}'$ be our $\mathscr{B}^*$. Otherwise note that for any maximally-consistent $\mathscr{B}^*\supseteq\mathscr{B}'$ we have $\mathfrak{R}_a\mathscr{B}^*\mathscr{C}^*$ because of condition (c) again and because $\mathscr{C}^*$ is maximally-consistent. Finally, since $\nVdash_{\mathscr{C}^*,\mathfrak{R}_A} \psi$, we have $\nVdash_{\mathscr{B}^*,\mathfrak{R}_A} K_a\psi$.\footnote{It might be of interest to note here that this proof relies on $\mathfrak{R}_a$ being symmetric. This suggests that it cannot be used to show that this property holds for arbitrary modal relations. More conditions on modal relations might be required for this lemma to hold for non-symmetric modal relations.}
\end{proof}

The following interesting corollary follows immediately from the special case in which $\mathscr{B}$ is the empty base.

\begin{cor}
    Any formula $\phi$ that is valid at all maximally-consistent bases for all $\mathfrak{R}_A$, is valid.
\end{cor}

\iffalse{
As mentioned above, our strategy for proving soundness is based on the idea that we can show an equivalence between maximally-consistent bases and worlds in Kripke semantics. A problem with this approach, however, is that a model might have multiple worlds with the same valuations while there can only be one corresponding maximally-consistent base for those worlds. Lemma \ref{lem:ModalAlmostMaxCon} shows that the bases right below the maximally-consistent bases agree with their maximally-consistent superset bases on the validity of all formulae. This is because they already have enough rules to determine the validity for any formula and adding new rules only makes a difference, if it causes them to become inconsistent.

\begin{lemma}
    \label{lem:ModalAlmostMaxCon}

    For all $\phi$, $\mathfrak{R}_A$, and bases $\mathscr{B}$ and $\mathscr{B}'$ s.t. $\mathscr{B}$ is the only maximally-consistent superset base of $\mathscr{B}'$,

    $\Vdash_{\mathscr{B},\mathfrak{R}_A} \phi$ iff $\Vdash_{\mathscr{B}',\mathfrak{R}_A} \phi$
\end{lemma}

\begin{proof}
    The left-to-right direction follows immediately from Lemma \ref{ModalMonotonicity}. For the right-to-left direction, assume $\nVdash_{\mathscr{B},\mathfrak{R}_A} \phi$ and, since $\mathscr{B}'$ is the only maximally-consistent superset base of $\mathscr{B}$, $\nVdash_{\mathscr{B}',\mathfrak{R}_A} \phi$ follows from Lemma \ref{lem:nonvalidMaxcon}.
\end{proof}
}\else

\section{Soundness and completeness}
\label{sec:completeness}
We show completeness between our semantics and the axiomatic system for $S5$-modal logic step-by-step. We first show that the axioms and rules of classical logic are valid in our semantics and then add the modal axioms and necessitation. This method of proving completeness differs from the methods generally used in the literature (see, for example, \cite{Sandqvist2005,Sandqvist2015}) in which completeness is shown between a natural deduction proof system and the base-extension semantics.\footnote{Recall that the result we call completeness is called soundness in the cited papers, because of our use of Makinson's \cite{Makinson2014} choice of soundness and completeness.} 

We believe it is fitting to focus on the axiomatic system in the case of modal logics because of the close connection between a modal logic and its corresponding modal axiom(s).\footnote{Note that, in this paper, we are concerned --- as in, for example, \cite{Sandqvist2015} --- with the use of base-extension semantics to define the validity of formulae not the validity of proofs. The latter occurs in a different aspect of proof-theoretic semantics in which the meanings of the connectives are determined by various analyses of Gentzen-style systems. See \cite{Schroeder2007modelvsproof} for a helpful introductory discussion of the latter.} 
We follow the same approach as in the appendix of \cite{Eckhardt2024}, in which we have shown the completeness of our base-extension semantics for single-agent modal logics up to $S4$ with respect to their corresponding axiomatic proof systems.

We start by showing that \emph{(MP)} and the classical axioms hold.

\begin{lemma} \label{lem:classicalaxioms}
    The following (classical) axioms hold:

    \begin{enumerate}[label=(\arabic{enumi})]
         
        \item $\phi\Vdash \psi\to\phi$ 
        \item $\phi\to(\psi\to\chi)\Vdash (\phi\to\psi)\to(\phi\to\chi)$
        \item $(\phi\to\bot)\to(\psi\to\bot)\Vdash \psi\to\phi$.
    \end{enumerate}

    Additionally, the following rule holds:

        \begin{enumerate}[label=(MP)]
            \item From $\Vdash \phi\to \psi$ and $\Vdash \phi$, infer $\Vdash \psi$.       
        
        \end{enumerate}

\end{lemma}

\begin{proof}
    The detailed proof can be found in the appendix. We give a sketch here.

    The proof of \emph{(MP)} uses the validity condition of $\to$ from Definition \ref{EXTValidity} to proof that for all $\mathscr{B}$ and $\mathfrak{R}_A$
    
    \begin{center}
    If $\Vdash_{\mathscr{B},\mathfrak{R}_A} \phi\to \psi$ and $\Vdash_{\mathscr{B},\mathfrak{R}_A} \phi$, then $\Vdash_{\mathscr{B},\mathfrak{R}_A} \psi$.
    \end{center}

    \noindent \emph{(MP)} follows immediately from this result.
    
    The proofs of \emph{(1)} and \emph{(2)} follow the same strategy: For an arbitrary base $\mathscr{B}$ and set of relations $\mathfrak{R}_A$, we assume that the formula on the left of the turnstyle holds and show that the formula on the right follows. As we are trying to show conditional formulae, we additionally assume the antecedent and show that the consequent has to hold at all supersets of $\mathscr{B}$ using Lemma \ref{ModalMonotonicity}.

    For the proof of \emph{(3)}, we show that at any maximally-consistent superset base of $\mathscr{B}$, called $\mathscr{C}$,  at which either $\phi$ or $\phi\to\bot$ holds, $\psi\to\phi$ follows. Given Lemma \ref{lem:excludedmiddle}, we know that $\phi$ or $\phi\to\bot$ has to hold for all maximally-consistent bases and so $\psi\to\phi$ holds at all $\mathscr{C}$. From the contrapositive of Lemma \ref{lem:nonvalidMaxcon}, it follows that $\psi\to\phi$ holds at $\mathscr{B}$ as well.
\end{proof}

Lemma \ref{lem:classicalaxioms} tells us that all classical tautologies will also be valid in our base-extension semantics. This gives us completeness of our system to classical logic. What is left to show is that the same holds for the corresponding $S5$ modal formulae. We prove this analogously by showing that the axioms and rules of our modal logic hold at our bases.

\begin{lemma} \label{lem:modalaxioms}
    The following (modal) axioms hold:

    \begin{enumerate}[label=]
         \item (K) $K_a(\phi\to\psi)\Vdash K_a\phi\to K_a\psi$
         \item (T)  $K_a\phi\Vdash \phi$
         \item (4) $K_a \phi\Vdash K_a K_a\phi$.  
         \item (5) $\neg K_a \phi\Vdash K_a\neg K_a\phi$

    \end{enumerate}

    Additionally, the following rule holds:

        \begin{enumerate}[label=]
            \item (NEC) From $\Vdash \phi$, infer $\Vdash K_a\phi$       
        
        \end{enumerate}

\end{lemma}

\begin{proof} 
    The detailed proof can be found in the appendix. We give a sketch here.
    
    The proof of these goes via contradiction with the exception of \emph{(T)}, which can be straightforwardly proven to follow directly from the reflexivity of $\mathfrak{R}_a$.

    For the other cases, we take some base $\mathscr{B}$ and set of relations $\mathfrak{R}_A$ and assume the formula to the left of the turnstyle. We then show that the right's not holding leads to a contradiction given that $\mathfrak{R}_a$ is a $S5$-modal relation. The arguments used closely mirror equivalent arguments in Kripke semantics. 
\end{proof}

From the Lemmas \ref{lem:classicalaxioms} and \ref{lem:modalaxioms}, our completeness result follows immediately.

\begin{theorem}[Completeness]  If $\phi$ is a theorem of $S5$ modal logic, then $\phi$ is valid in the base-extension semantics for $S5$.
\end{theorem}

We show soundness with respect to the Kripke semantics for epistemic logic rather than the axiomatic system as we did for completeness. We take a formula that is not valid in Kripke semantics and show that is also not valid on the corresponding base-extension semantics. This suffices as the Kripke semantics is sound and complete with respect to the axiomatic system.


This proof-strategy is the same as the one used to show soundness in \cite{Eckhardt2024}. 

\begin{theorem}[Soundness] If $\phi$ is valid in the base-extension semantics for multi-agent modal logic $S5$, then $\phi$ is a theorem of multi-agent modal logic $S5$.
\end{theorem}

\begin{proof}
    The detailed proof can be found in the appendix. We give a sketch here.

Take some $\phi$ that is not valid in the Kripke semantics of multi-agent $S5$, we know that there is a model and a world s.t. $M,w\nvDash \phi$. We would like to use that model to create bases and a set of modal relations by taking the maximally-consistent bases that correspond to those worlds. However this would not work, because it is possible that $M$ contains multiple worlds with the same valuation, but we cannot have multiple maximally-consistent bases for which the same formulae hold. So, first we tweak $M$ to create a model $M'$ s.t. no two worlds in $M'$ have the same valuation but for which we still have $M',w\nvDash \phi$ by changing the valuations for propositional formulae that do not appear in $\phi$. We then use the methods from \cite{Makinson2014} to construct maximally-consistent bases from the worlds in $M'$ by choosing our base rules in a way that guarantees that all non-modal formulae are valid at the bases iff they are true at the worlds. Next we construct a set of relations $\mathfrak{R}_A$ on our set of bases that, for the bases corresponding to worlds, mimics the relations $R_a$ of $M'$. This concludes the construction of our counterexample.

What is left is to show that it is in fact the counterexample we need. To do so we show that the relations in $\mathfrak{R}_A$ are $S5$-modal relations (i.e., they are reflexive, transitive, and Euclidean, and that the conditions of Definition \ref{modalRelation} hold) and that, for the base $\mathscr{B}_w$ that corresponds to $w$ in $M'$, any formula holds at $\mathscr{B}_w$ iff they are true at $w$. Since $\mathscr{B}_w$ is a maximally-consistent base we can use Lemma \ref{ModalBehaviour}. From this and $M,w\nvDash \phi$, it follows that $\nVdash_{\mathscr{B}_w, \mathfrak{R}_A} \phi$ and so, we can conclude that $\phi$ is not valid in base-extension semantics.
\end{proof}

Here is a list of the steps of the proof (adapted from \cite{Eckhardt2024}). 

\begin{enumerate}
    \item Construct $M'$ from $M$ s.t. there are no two worlds with the same valuation.
    \item Construct maximally-consistent bases corresponding to the worlds in $M'$.
    \item Define a set of relations $\mathfrak{R}_A$ on the set of bases using $R_A$ of $M'$.
    \item Prove that the members of $\mathfrak{R}_A$ are $S5$-modal relations.
    \item Prove that $\Vdash_{\mathscr{B}_w, \mathfrak{R}_A} \psi$ iff $M',w\vDash \psi$, for all $\psi$. 
    \item Since $M,w\nvDash \phi$, it follows that $\nVdash_{\mathscr{B}_w, \mathfrak{R}_A} \phi$. So, $\phi$ is not valid in the base-extension semantics for epistemic logic.
\end{enumerate}

\section{Discussion} \label{sec:discussion}

We have developed a base-extension semantics for the multi-agent modal logic $S5$. In \cite{Eckhardt2024}, a base-extension semantics for modal logic has been given that is not complete with respect to Euclidean models. We have generalized the semantics to multiple agents and adapted the conditions for modal relations. A modal semantics with an unindexed $\square$ can be retrieved simply by considering the case in which we only have a single agent. We have established that the resulting semantics is sound and complete.


The semantics for multi-agent $S5$ modal logic provides the semantics for static epistemic logic from which, in future work, the authors plan to establish a semantics for PAL. In Kripke semantics, PAL is given in terms of an update semantics that takes an epistemic model and updates it corresponding to the information obtained from a certain announcement. In base-extension semantics, however, such an update has to be applied to the modal relations. 

For a simple example, take an $S5$-model with two worlds $w$ and $v$ and a single agent $a$ s.t. $R_a wv$. Let some formula $\phi$ hold at $w$ but not at $v$. An announcement of $\phi$ should result in a model in which $R_a wv$ is no longer the case as $a$ no longer considers non-$\phi$ worlds possible.  Similarly, in base-extension semantics an announcement of $\phi$ should change $\mathfrak{R}_a$ to no longer hold between $\phi$ and non-$\phi$ bases. This would, however, potentially cause $\mathfrak{R}_a$ to no longer be a modal relation, as non-$\phi$ bases can have superset bases on which $\phi$ holds and, conversely, non-$\phi$ be the case on subset bases of $\phi$ bases. 


We explore an approach that ensures that the resulting relations are modal relations by differentiating at which base an announcement has been taken: If an announcement of $\phi$ at a base $\mathscr{B}$ causes $\mathscr{B}$ to no longer have access to a non-$\phi$ base $\mathscr{C}$, then all sub- and super-sets of $\mathscr{B}$ do not have access to those sub- and super-sets of the $\mathscr{C}$ that they had access to before. The details of this approach remain future work.

This provides the first step into the broader world of dynamic epistemic logics.

\bibliographystyle{plain}
	\bibliography{DELpts}

\appendix 

\section{Proof Details Elided in Main Text}

\medskip

\noindent{\bf Lemma 5.1} \emph{The following axioms hold:}
    \begin{itemize} 
        \item[\emph{(1)}] $\phi\Vdash\psi\to\phi$ 
        \item[\emph{(2)}] $\phi\to(\psi\to\chi)\Vdash (\phi\to\psi)\to(\phi\to\chi)$
        \item[\emph{(3)}] $(\phi\to\bot)\to(\psi\to\bot)\Vdash \psi\to\phi$.
    \end{itemize}
    \emph{Additionally, the following rule holds:}
        \begin{enumerate}[label=$(MP)$]
            \item \emph{From $\Vdash \phi\to \psi$ and $\Vdash \phi$, infer $\Vdash\psi$.}       
        \end{enumerate}

\begin{proof}
    We start by showing \emph{(MP)} as it is helpful in establishing the other results. In fact, we show that for all $\mathscr{B}$ and $\mathfrak{R}_A$,

    \begin{center}
        \emph{(MP*)} if $\Vdash_{\mathscr{B},\mathfrak{R}_A} \phi\to \psi$ and $\Vdash_{\mathscr{B},\mathfrak{R}_A} \phi$, $\Vdash_{\mathscr{B},\mathfrak{R}_A} \psi$.
    \end{center}

    \noindent \emph{(MP)} follows immediately from \emph{(MP*)}. 
    
    Simply note that because $\Vdash_{\mathscr{B},\mathfrak{R}_A} \phi\to\psi$, we have for all $\mathscr{C}\supseteq\mathscr{B}$, if $\Vdash_{\mathscr{C},\mathfrak{R}_A} \phi$, then $\Vdash_{\mathscr{C},\mathfrak{R}_A} \psi$. By assumption, $\Vdash_{\mathscr{B},\mathfrak{R}_A} \phi$ and so $\Vdash_{\mathscr{B},\mathfrak{R}_A} \psi$.

    By Definition \ref{EXTValidity}, for \emph{(1)} we need to show that for any $\mathscr{B}$, if $\Vdash_{\mathscr{B},\mathfrak{R}_A} \phi$, then $\Vdash_{\mathscr{B},\mathfrak{R}_A} \psi\to\phi$. By Lemma \ref{ModalMonotonicity}, for any $\mathscr{C}\supseteq\mathscr{B}$ s.t. $\Vdash_{\mathscr{C},\mathfrak{R}_A} \psi$ we also have $\Vdash_{\mathscr{C},\mathfrak{R}_A} \phi$ and so $\Vdash_{\mathscr{B},\mathfrak{R}_A} \psi\to\phi$.

    For \emph{(2)} we follow the same strategy of showing that if the formula on the left of the turnstyle holds at a base, so does the right. To show that $\Vdash_{\mathscr{B},\mathfrak{R}_A} (\phi\to\psi)\to(\phi\to\chi)$, we need to show that for all $\mathscr{C}\supseteq\mathscr{B}$, if $\Vdash_{\mathscr{C},\mathfrak{R}_A} \phi\to\psi$, then $\Vdash_{\mathscr{C},\mathfrak{R}_A} \phi\to\chi$. Take some $\mathscr{D}\supseteq\mathscr{C}$ s.t. $\Vdash_{\mathscr{D},\mathfrak{R}_A} \phi\to(\psi\to\chi)$, $\Vdash_{\mathscr{D},\mathfrak{R}_A} \phi\to\psi$ and $\Vdash_{\mathscr{D},\mathfrak{R}_A} \phi$. By \emph{(MP*)}, we get $\Vdash_{\mathscr{D},\mathfrak{R}_A} \chi$ and so, $\Vdash_{\mathscr{C},\mathfrak{R}_A} \phi\to\chi$ and, finally, $\Vdash_{\mathscr{B},\mathfrak{R}_A} (\phi\to\psi)\to(\phi\to\chi)$.

    For \emph{(3)} we take some $\mathscr{B}$ with $\Vdash_{\mathscr{B},\mathfrak{R}_A} (\phi\to\bot)\to(\psi\to\bot)$. To show that $\Vdash_{\mathscr{B},\mathfrak{R}_A} \psi\to\phi$, we show that for all $\mathscr{C}\supseteq\mathscr{B}$ s.t. $\Vdash_{\mathscr{C},\mathfrak{R}_A} \psi$ we also have $\Vdash_{\mathscr{C},\mathfrak{R}_A} \phi$. Assume $\nVdash_{\mathscr{C},\mathfrak{R}_A} \phi$ for contradiction. By Lemma \ref{lem:excludedmiddle}, all maximally-consistent $\mathscr{C}^*\supseteq\mathscr{C}$ either $\Vdash_{\mathscr{C}^*,\mathfrak{R}_A} \phi$ or $\Vdash_{\mathscr{C}^*,\mathfrak{R}_A} \phi\to\bot$. By Lemma \ref{ModalMonotonicity}, we know that $\Vdash_{\mathscr{C}^*,\mathfrak{R}_A} (\phi\to\bot)\to(\psi\to\bot)$ and $\Vdash_{\mathscr{C}^*,\mathfrak{R}_A} \psi$. So, if $\Vdash_{\mathscr{C}^*,\mathfrak{R}_A} \phi\to\bot$, then, by \emph{(MP*)}, $\Vdash_{\mathscr{C}^*,\mathfrak{R}_A} \psi\to\bot$ and $\Vdash_{\mathscr{C}^*,\mathfrak{R}_A} \bot$. This is a contradiction as $\mathscr{C}^*$ is, by assumption, consistent. So, $\phi$ holds at all maximally-consistent superset bases of $\mathscr{C}$ and so, by the contrapositive of Lemma \ref{lem:nonvalidMaxcon}, $\Vdash_{\mathscr{C},\mathfrak{R}_A} \phi$.
\end{proof}

\medskip 

\noindent{\bf Lemma 5.2} \emph{The following axioms hold:}
    \begin{enumerate}[label=]
         \item $(K)$ $K_a(\phi\to\psi)\Vdash K_a\phi\to K_a\psi$
         \item $(T)$  $K_a\phi\Vdash \phi$
         \item $(4)$ $K_a \phi\Vdash K_a K_a\phi$.  
         \item $(5)$ $\neg K_a \phi\Vdash K_a\neg K_a\phi$
    \end{enumerate} 
    \emph{Additionally, the following rule holds:}
        \begin{enumerate}[label=]
            \item $(NEC)$ \emph{From $\Vdash \phi$, infer $\Vdash K_a\phi$}       
        \end{enumerate}

\begin{proof}
    For \emph{(K)}, we show the contrapositive and assume $\nVdash_{\mathscr{B},\mathfrak{R}_A} K_a\phi\to K_a\psi$ for an arbitrary $\mathscr{B}$ and $\mathfrak{R}_A$. So, there is a $\mathscr{C}\supseteq\mathscr{B}$ s.t. $\Vdash_{\mathscr{C},\mathfrak{R}_A} K_a\phi$ and $\nVdash_{\mathscr{C},\mathfrak{R}_A} K_a\psi$. So, there are $\mathscr{D}\supseteq\mathscr{C}$ and $\mathscr{E}$ s.t. $\mathfrak{R}_a\mathscr{D}\mathscr{E}$ with $\Vdash_{\mathscr{E},\mathfrak{R}_A} \phi$ but $\nVdash_{\mathscr{E},\mathfrak{R}_A} \psi$ and so $\nVdash_{\mathscr{E},\mathfrak{R}_A} \phi\to\psi$ and $\Vdash_{\mathscr{D},\mathfrak{R}_A} K_a(\phi\to\psi)$. Finally, as $\mathscr{D}\supseteq\mathscr{B}$, $\nVdash_{\mathscr{B},\mathfrak{R}_A} K_a(\phi\to\psi)$.  

    For \emph{(T)}, take some $\Vdash_{\mathscr{B},\mathfrak{R}_A} K_a\phi$. By the reflexivity of $\mathfrak{R}_a$, $\mathfrak{R}_a\mathscr{B}\mathscr{B}$ and so $\Vdash_{\mathscr{B},\mathfrak{R}_A} \phi$.

For \emph{(4)}, we, without loss of generality, assume some $\mathscr{B}$ and $\mathfrak{R}_A$ s.t. $\Vdash_{\mathscr{B}, \mathfrak{R}_A} K_a \phi$. For contradiction, we also assume $\nVdash_{\mathscr{B}, \mathfrak{R}_A} K_a K_a\phi$. So, there are $\mathscr{C}\supseteq\mathscr{B}$ and $\mathscr{D}$ s.t. $\mathfrak{R}_a\mathscr{C}\mathscr{D}$ and $\nVdash_{\mathscr{D}, \mathfrak{R}_A} K_a \phi$. This means there have to be $\mathscr{E}\supseteq\mathscr{D}$ and $\mathscr{F}$ s.t. $\mathfrak{R}_a\mathscr{E}\mathscr{F}$ and $\nVdash_{\mathscr{F}, \mathfrak{R}_A} \phi$. By Definition \ref{modalRelation} $(d)$, we know there is a $\mathscr{G}\subseteq\mathscr{F}$ and $\mathfrak{R}_a \mathscr{D}\mathscr{G}$ and, by transitivity, $\mathfrak{R}_a\mathscr{C}\mathscr{G}$. Obviously, $\nVdash_{\mathscr{G}, \mathfrak{R}_A} \phi$ and so $\nVdash_{\mathscr{B}, \mathfrak{R}_A} K_a \phi$, which contradicts our assumption.
    
    For \emph{(5)}, we assume, without loss of generality, some $\mathscr{B}$ and $\mathfrak{R}_A$ s.t. $\Vdash_{\mathscr{B}, \mathfrak{R}_A} \neg K_a \phi$ and, for contradiction, $\nVdash_{\mathscr{B}, \mathfrak{R}_A} K_a\neg K_a \phi$. So, there is $\mathscr{C}$ s.t. $\mathfrak{R}_a\mathscr{B}\mathscr{C}$ and $\nVdash_{\mathscr{C}, \mathfrak{R}_A} \neg K_a \phi$. So, there is a $\mathscr{D}\supseteq\mathscr{C}$ s.t $\mathscr{D}$ is consistent and $\Vdash_{\mathscr{D}, \mathfrak{R}_A} K_a \phi$.

    From $\Vdash_{\mathscr{B}, \mathfrak{R}_A} \neg K_a \phi$ it follows that for all consistent $\mathscr{E}\supseteq\mathscr{B}$, we have $\nVdash_{\mathscr{B}, \mathfrak{R}_A} K_a \phi$. So there is a $\mathscr{F}$ s.t. $\mathfrak{R}_a\mathscr{E}\mathscr{F}$ and $\nVdash_{\mathscr{F}, \mathfrak{R}_A} \phi$.

    By symmetry, we have $\mathfrak{R}_a \mathscr{C}\mathscr{B}$ and so, by condition (c) of Definition \ref{modalRelation}, for one of these $\mathscr{E}\supseteq\mathscr{B}$ we have $\mathfrak{R}_a\mathscr{D}\mathscr{E}$ and, again by symmetry, $\mathfrak{R}_a \mathscr{E}\mathscr{D}$. By Euclidean and $\mathfrak{R}_a\mathscr{E}\mathscr{F}$, we have $\mathfrak{R}_a\mathscr{D}\mathscr{F}$ which gives us our contradiction: We have $\Vdash_{\mathscr{D}, \mathfrak{R}_A} K_a \phi$ and $\nVdash_{\mathscr{F}, \mathfrak{R}_A} \phi$.

    \iffalse{
    
    For (5), we assume, without loss of generality, some $\mathscr{B}$ and $\mathfrak{R}$ s.t. $\Vdash^\gamma_{\mathscr{B}, \mathfrak{R}} \lozenge \phi$ and, for contradiction, $\nVdash^\gamma_{\mathscr{B}, \mathfrak{R}} \square\lozenge \phi$. So, there are $\mathscr{C}\supseteq\mathscr{B}$ and $\mathscr{D}$ s.t. $\mathfrak{R}\mathscr{C}\mathscr{D}$ and $\nVdash^\gamma_{\mathscr{D}, \mathfrak{R}} \lozenge \phi$. So, there are $\mathscr{E}\supseteq\mathscr{D}$ and for all $\mathscr{F}$ s.t. $\mathfrak{R}\mathscr{E}\mathscr{F}$ $\nVdash^\gamma_{\mathscr{F}, \mathfrak{R}} \phi$.

    \noindent By $\Vdash^\gamma_{\mathscr{B}, \mathfrak{R}} \lozenge \phi$, we know there is a $\mathscr{G}$ s.t. $\mathfrak{R}\mathscr{C}\mathscr{G}$ and $\Vdash^\gamma_{\mathscr{G}, \mathfrak{R}} \phi$ and, by Euclidean, $\mathfrak{R}\mathscr{D}\mathscr{G}$. By (c), we know that there is an $\mathscr{H}\supseteq\mathscr{G}$ s.t. $\mathfrak{R}\mathscr{E}\mathscr{H}$ and, by Lemma \ref{ModalMonotonicity}, $\Vdash^\gamma_{\mathscr{H}, \mathfrak{R}} \phi$. So, we have a contradiction.

DIAMOND!
}\else

    We show \emph{(NEC)} via the contrapositive. If $\nVdash^A K_a\phi$, then $\nVdash^A \phi$. Suppose $\nVdash^A K_a\phi$. So, there is a $\mathscr{B}$ and an $\mathfrak{R}_A$ s.t. $\nVdash_{\mathscr{B},\mathfrak{R}_A} K_a\phi$. Therefore, there are $\mathscr{C}\supseteq\mathscr{B}$ and $\mathscr{D}$ s.t. $\mathfrak{R}_a\mathscr{C}\mathscr{D}$ and $\nVdash_{\mathscr{D},\mathfrak{R}_A} \phi$ and, so, $\nVdash^A\phi$.
    
\end{proof}

\medskip 

\noindent{\bf Theorem 5.4 (Soundness)} \emph{If $\phi$ is valid in the base-extension semantics for multi-agent modal logic $S5$, then $\phi$ is a theorem of multi-agent modal logic $S5$.}
\begin{proof}
    We assume an $S5$ model $M = \langle F,V\rangle$ with $F = \langle W,R_A\rangle$, $R_A$ is a set of relation $R_a$, and a world $w$ s.t. $M,w\nvDash \phi$. We now implement the steps of the proof sketch.
    
    We begin with step (i). We can have worlds with the same valuation but not maximally-consistent bases that agree on all atomic formulae. So we need to adapt our $M$ so that every world disagrees on some atomic formulae that are not relevant to $\phi$. We do so by assigning a different atomic formula $q$ that does not appear in $\phi$ to each world. This is possible as we have infinitely many atomic formulae. For every $v\in W$, let $q_v$ be some atomic formula s.t. it does not appear in $\phi$ and, for all worlds $u$, if $u\neq v$, then $q_v \neq q_u$.

    Let $M' = \langle F, V'\rangle$ be s.t. for all $v\in W$, $V'(q_v) = W\setminus\{v\}$ and $V'(p) = V(p)$ for all other atomic formulae $p$. Note that $M', w \nvDash \phi$. This concludes our first step.

    In step (ii), this model is used to construct bases that correspond to its worlds. For every world $w\in W$, we define a base $\mathscr{A}_w$ in the following way:

    \[
    \begin{split}
    \mathscr{A}_w := & \{\Rightarrow p : M',w\vDash p\} \cup \\
                 & \{\Rightarrow q_v : v \in W \text{ and } v\neq w\} \cup \\
                 & \{p\Rightarrow q_w, q_w\Rightarrow p : M',w\nvDash p\} \\
    \end{split}
    \]

    By Lemma \ref{MaxCon}, we know there is a maximally-consistent $\mathscr{C}\supseteq\mathscr{A}_w$ s.t. $\Vdash_{\mathscr{C}} q_w$. We call this $\mathscr{C}$ $\mathscr{B}_w$.

    This gives us the bases we require, but we still need to define a set of relations $\mathfrak{R}_A$ between them. This is step (iii). 

    We consider the relations such that, for every agent $a\in A$, $\mathfrak{R}_a\mathscr{B}\mathscr{C}$ only by the following:

    \begin{enumerate}[label=(\arabic{enumi})]
    \item if $R_a wv$, then $\mathfrak{R}_a\mathscr{B}_w\mathscr{B}_v$
    \item if $\mathscr{B}$ and $\mathscr{C}$ are inconsistent, then $\mathfrak{R}_a\mathscr{B}\mathscr{C}$
    
    \item 
    \begin{enumerate}
    \item for all $\mathscr{B}$ and $\mathscr{C}$ s.t. there are exactly one maximally-consistent $\mathscr{B}'\supset\mathscr{B}$ and one maximally-consistent $\mathscr{C}'\supset \mathscr{C}$ and there are $w, v\in W$ s.t. $\mathscr{B}'=\mathscr{B}_w$ and $\mathscr{C}'=\mathscr{B}_v$, if $R_a wv$, then $\mathfrak{R}_a\mathscr{B}\mathscr{C}$, 

    \item for all $\mathscr{B}$, $\mathscr{C}$ s.t. there are $w,v \in W$ and $\mathscr{B}_w\supset \mathscr{B}$ and $\mathscr{B}_v\supset\mathscr{B}$, if $R_a wv$ and there are maximally-consistent $\mathscr{B}'\supset\mathscr{B}$ and $\mathscr{C}'\supset\mathscr{C}$ and $\mathscr{B}'\neq\mathscr{B}_w$ and $\mathscr{C}'\neq\mathscr{B}_v$, then $\mathfrak{R}_a\mathscr{B}\mathscr{C}$,

    \item for all consistent $\mathscr{B}$, if there are no $w,v\in W$ with $\mathscr{B}_w\supseteq\mathscr{B}$ and $\mathscr{B}_v\supseteq\mathscr{C}$, then $\mathfrak{R}_a\mathscr{B}\mathscr{C}$,
    \end{enumerate}

    \item reflexivity: for all bases $\mathscr{B}$, $\mathfrak{R}_a\mathscr{B}\mathscr{B}$, 

    \item transitivity: if $\mathfrak{R}_a\mathscr{B}\mathscr{C}$ and $\mathfrak{R}_a\mathscr{C}\mathscr{D}$, then $\mathfrak{R}_a\mathscr{B}\mathscr{D}$

    \item Euclidean: if $\mathfrak{R}_a\mathscr{B}\mathscr{C}$ and $\mathfrak{R}_a\mathscr{B}\mathscr{D}$, then $\mathfrak{R}_a\mathscr{C}\mathscr{D}$.
    
\end{enumerate}

    We have now constructed the bases and the relation required, but we have not yet shown that the relations are in fact $S5$-modal relations. This is step (iv).

    It is easy to see that any $\mathfrak{R}_a$ is an equivalence relation by (4), (5), and (6). We still need to show that the conditions (a)-(d) from Definition \ref{modalRelation} hold (i.e., that it is a modal relation). Since (2) simply connects all inconsistent bases with each other, condition $(a)$ holds straightforwardly. Any step that can connect a base to an inconsistent base can only do so if the initial base is itself inconsistent. So, $(b)$ follows.

To show (c), we have to consider all steps of the construction of $\mathfrak{R}_a$ that introduce new connections. (1) only applies for maximally-consistent bases and (2) only for inconsistent ones, so (c) holds for those trivially. 

For (3c), note that all $\mathscr{D}\supseteq \mathscr{B}$ are also do not have any $w$ s.t. $\mathscr{B}_w \supseteq \mathscr{D}$ and so $\mathfrak{R}_a\mathscr{D}\mathscr{C}$. For $\mathfrak{R}_a\mathscr{B}\mathscr{C}$ added by (3a), there are $\mathscr{B}_w\supset\mathscr{B}$ and $\mathscr{B}_v\supset\mathscr{C}$ with $\mathfrak{R}_a\mathscr{B}_w\mathscr{B}_v$ and for all $\mathscr{B}\subseteq\mathscr{B}'\subset\mathscr{B}_w$ we have $\mathfrak{R}_a\mathscr{B}'\mathscr{C}$.   

For (3b), if there are $w,v \in W$ and $\mathscr{B}_w\supseteq \mathscr{B}$ and $\mathscr{B_v}\supseteq\mathscr{B}$ and there are maximally-consistent $\mathscr{B}'\supseteq\mathscr{B}$ and $\mathscr{C}'\supseteq\mathscr{C}$ and $\mathscr{B}'\neq\mathscr{B}_w$ and $\mathscr{C}'\neq\mathscr{B}_v$, then we have a couple of cases to consider for $\mathscr{D}$:

If $\mathscr{D}$ has $\mathscr{B}_w$ and another maximally-consistent base $\mathscr{D}'$ as their supersets, we get $\mathfrak{R}_a\mathscr{D}\mathscr{C}$ by (3b).

If $\mathscr{D} =\mathscr{B}_w$ for some $w$, then there is a $\mathscr{B}_v\supset\mathscr{C}$ and $\mathfrak{R}_a\mathscr{B}_w\mathscr{B}_v$.

If there is a $w$ and exactly one maximally-consistent $\mathscr{D}'\supset\mathscr{D}$ s.t. $\mathscr{D}'=\mathscr{B}_w$, then there is also a $v$ s.t. $\mathfrak{R}_a\mathscr{B}_w\mathscr{B}_v$ and $\mathscr{B}_v\supset\mathscr{C}$. By (3a), for all $\mathscr{E}\supset\mathscr{C}$ s.t. $\mathscr{B}_v$ is the only maximally-consistent superset of $\mathscr{E}$, we have $\mathfrak{R}_a\mathscr{D}\mathscr{E}$.

If $\mathscr{D}$ has no $w$ s.t. $\mathscr{B}_w\supseteq\mathscr{D}$, then we have $\mathfrak{R}_a\mathscr{D}\mathscr{E}$ for all $\mathscr{E}\supset\mathscr{C}$ s.t. there are no $v$ with $\mathscr{B}_v\supset\mathscr{E}$.

For (4), simply note that if $\mathfrak{R}_a\mathscr{B}\mathscr{B}$, then for all consistent $\mathscr{C}\supset\mathscr{B}$, we have $\mathfrak{R}_a\mathscr{C}\mathscr{C}$. For (5), take some $\mathfrak{R}_a$ with $\mathfrak{R}_a\mathscr{B}\mathscr{C}$ and $\mathfrak{R}_a\mathscr{C}\mathscr{D}$ for which (c) holds. We show that is still holds when we add $\mathfrak{R}_a \mathscr{B}\mathscr{D}$. By assumption we know that for all $\mathscr{B}'\supseteq\mathscr{B}$ there are $\mathscr{C}'\supseteq\mathscr{C}$ s.t. $\mathfrak{R}_a\mathscr{B}'\mathscr{C}'$ by (c). Similarly, we know that there is a $\mathscr{D}'\supseteq\mathscr{D}$ and that $\mathfrak{R}_a\mathscr{C}'\mathscr{D}'$. So, by transitivity, $\mathfrak{R}_a\mathscr{B}'\mathscr{D}'$. For (6), we follow the same strategy. Suppose $\mathfrak{R}_a\mathscr{B}\mathscr{C}$ and $\mathfrak{R}_a\mathscr{B}\mathscr{D}$. By assumption we know that for all $\mathscr{B}'$ there are $\mathscr{C}'\supseteq\mathscr{C}$ and $\mathscr{D}'\supseteq\mathscr{D}$ with $\mathfrak{R}_a\mathscr{B}'\mathscr{C}'$ and $\mathfrak{R}_a\mathscr{B}'\mathscr{D}'$ and so, by Euclidean, $\mathfrak{R}_a\mathscr{C}'\mathscr{D}'$. With this we have shown that (c) holds for $\mathfrak{R}_a$.

For (d), again, we go through the steps of the construction of $\mathfrak{R}_a$. Again, (2) holds trivially because it only deals with inconsistent bases.

For $(3b)$, it suffices that for every $\mathscr{B}'\subseteq\mathscr{B}$ we have $\mathfrak{R}_a\mathscr{B}'\emptyset$.

For $(3c)$ and $\mathscr{B}'\subseteq\mathscr{B}$, there are two cases to consider: Either there is a $w$ s.t. $\mathscr{B}'\subset\mathscr{B}_w$, in which case, by $(3b)$, $\mathfrak{R}_a\mathscr{B}'\emptyset$, or there is no such $w$, in which case for all $\mathscr{C}$, s.t. $\mathfrak{R}_a\mathscr{B}\mathscr{C}$, we also have $\mathfrak{R}_a\mathscr{B}'\mathscr{C}$, by $(3c)$.

For $(3a)$ and $\mathscr{B}'\subseteq\mathscr{B}$, there are similarly two cases to consider. If $\mathscr{B}_w$ is still the only maximally-consistent superset of $\mathscr{B}'$, then,  by $(3a)$, $\mathfrak{R}_a\mathscr{B}'\mathscr{C}$ for all $\mathscr{C}$ s.t. $\mathfrak{R}_a\mathscr{B}\mathscr{C}$. Otherwise, by $(3b)$, $\mathfrak{R}_a\mathscr{B}'\emptyset$.

For $(1)$ we have to consider three cases now: The case in which $\mathscr{B}'=\mathscr{B}_w$ is trivial, the case in which $\mathscr{B}_w$ is the only maximally-consistent superset base of $\mathscr{B}'$ immediately follows from $(3a)$, and, finally, if there are multiple maximally-consistent supersets of $\mathscr{B}'$, then, by $(3b)$, $\mathfrak{R}_a\mathscr{B}'\emptyset$.

For (4), (5), and (6), it suffices to take the proofs of (c) for these cases and flip the subset-relations. 

With these results, we have show that every relation $\mathfrak{R}_a\in \mathfrak{R}_A$ is an $S5$-modal relation. 

\iffalse{For (e), we show the contrapositive: For all $\mathscr{B}$ and $\mathscr{C}$ s.t. $\mathfrak{R}\mathscr{B}\mathscr{C}$ and for all formulae $\phi$, if $\nVdash^\gamma_{\mathscr{C},\mathfrak{R}} \phi$  then there are $\mathscr{D}\supset\mathscr{B}$ and $\mathscr{E}$ s.t. $\mathfrak{R}\mathscr{D}\mathscr{E}$ $\nVdash^\gamma_{\mathscr{E},\mathfrak{R}}\phi$. Again we go step-by-step through the construction of $\mathfrak{R}$.

For an inconsistent $\mathscr{B}$, note that we cannot have $\nVdash^\gamma_{\mathscr{C},\mathfrak{R}} \phi$ because $\mathscr{C}$ is inconsistent by (a). 

For $\mathscr{B}=\mathscr{B}_w$ for some $w$, note that $\mathscr{B}$ is maximally-consistent and so there are only inconsistent $\mathscr{D}$.

For $\mathscr{B}$ s.t. there is no $w$ with $\mathscr{B}\subseteq\mathscr{B}_w$, all consistent $\mathscr{D}\supset\mathscr{B}$ also have $\mathfrak{R}\mathscr{D}\mathscr{F}$ for all $\mathfrak{R}\mathscr{B}\mathscr{F}$, by (3c) and, especially, $\mathfrak{R}\mathscr{D}\mathscr{C}$.

For $\mathscr{B}$ s.t. there is a $w$ and $\mathscr{B}_w$ is the only maximally-consistent superset base of $\mathscr{B}$, if $\mathscr{D}$ is not maximally-consistent, then it also only has $\mathscr{B}_w$ as its maximally-consistent superset base and so $\mathfrak{R}\mathscr{D}\mathscr{C}$. Otherwise $\mathscr{D}=\mathscr{B}_w$ and there is a $v$ s.t. $\mathscr{E}=\mathscr{B}_v$ and we have $\nVdash^\gamma_{\mathscr{E},\mathfrak{R}} \phi$ by Lemma \ref{lem:ModalAlmostMaxCon}.


Finally, suppose $\mathscr{B}$ has multiple maximally-consistent supersets and there is at least one $w$ s.t. $\mathscr{B}_w\supseteq\mathscr{B}$. For simplicity, let us say that these base belong in group 3b, as their relations are covered by step (3b). Note that every base that is not in group 3b, is a superset of some of the 3b bases. Since all bases in 3b have access to all and only the 3b bases, it suffices to show the following: For all $\phi$, if there is a $\mathscr{B}$ in 3b s.t. $\nVdash^\gamma_{\mathscr{B},\mathfrak{R}} \phi$, then there is a $\mathscr{B}^*\supset\mathscr{B}$ s.t. $\mathscr{B}^*$ is not in 3b and $\nVdash^\gamma_{\mathscr{B}^*,\mathfrak{R}} \phi$. We proof this by induction on $\phi$.

For $\phi = p$ or $\bot$, this follows from Lemma \ref{MaxCon}.

For $\phi = \square\psi$, note that, by Definition \ref{EXTValidity}, $\nVdash^\gamma_{\mathscr{B},\mathfrak{R}} \square\psi$ iff there are $\mathscr{B}'\supseteq\mathscr{B}$ and $\mathscr{C}$ s.t. $\mathfrak{R}\mathscr{B}'\mathscr{C}$ and $\nVdash^\gamma_{\mathscr{C},\mathfrak{R}} \psi$. If $\mathscr{B}'$ is outside of 3b, we are done. Otherwise $\mathscr{C}$ has to also be in 3b and by the inductive hypothesis, there is a $\mathscr{C}'\supset\mathscr{C}$ with $\nVdash^\gamma_{\mathscr{C}^*,\mathfrak{R}} \psi$ outside of 3b. By our construction of $\mathfrak{R}$, there will also be a $\mathscr{B}^*\supset\mathscr{B}$ with $\mathfrak{R}\mathscr{B}^*\mathscr{C}^*$ and so $\nVdash^\gamma_{\mathscr{B}^*,\mathfrak{R}} \square\psi$.

Let $\phi= \psi\to\chi$. Again by Definition \ref{EXTValidity}, we know that $\nVdash^\gamma_{\mathscr{B},\mathfrak{R}} \psi\to\chi$ iff there is a $\mathscr{B}'\supseteq\mathscr{B}$ s.t. $\Vdash^\gamma_{\mathscr{B}',\mathfrak{R}} \psi$ but $\nVdash^\gamma_{\mathscr{B}',\mathfrak{R}} \chi$. Like before, if $\mathscr{B}'$ is outside of 3b, we are done. Otherwise there is a $\mathscr{B}^*\supset\mathscr{B}'$ outside of 3b with $\nVdash^\gamma_{\mathscr{B}^*,\mathfrak{R}} \chi$ by the inductive hypothesis. By Lemma \ref{ModalMonotonicity}, we know that $\Vdash^\gamma_{\mathscr{B}^*,\mathfrak{R}} \psi$ and we get $\nVdash^\gamma_{\mathscr{C},\mathfrak{R}} \psi\to\chi$.

This concludes our proof that condition (e) holds for $\mathfrak{R}$. As we have shown all conditions from Definition \ref{modalRelation}, we know that $\mathfrak{R}$ is a modal relation and so a $S5$-modal relation.
}\else

We proceed to show step (v): for all $\phi$ and $w\in W$, $\Vdash_{\mathscr{B}_w,\mathfrak{R}_A} \phi$ iff $M',w\vDash \phi$. We do so by induction on $\phi$. The case in which $\phi= p$ follows directly from our construction of $\mathscr{B}_w$. 

For the case in which $\phi=\bot$ it suffices to point out that $\bot$ cannot hold at $w$ and that $\mathscr{B}_w$ is consistent. 

For $\phi=\psi\to\chi$, note that $M',w\Vdash \psi\to\chi$ iff $M',w\Vdash \psi$ or $M,w\nVdash \chi$. By Lemma \ref{ModalBehaviour}, we similarly have $\Vdash_{\mathscr{B}_w,\mathfrak{R}_A} \psi\to\chi$ iff $\nVdash_{\mathscr{B}_w,\mathfrak{R}_A} \psi$ or $\Vdash_{\mathscr{B}_w,\mathfrak{R}_A} \chi$. We conclude by the inductive hypothesis.

Finally, for $\phi= K_a \psi$, $M',w\Vdash K_a \psi$ iff for all $v$ s.t. $R_a wv$, $M',v\Vdash \psi$. Again by Lemma \ref{ModalBehaviour}, we have $\Vdash_{\mathscr{B}_w,\mathfrak{R}_A} K_a \psi$ iff for all $\mathscr{C}$ s.t. $\mathfrak{R}_a\mathscr{B}_w\mathscr{C}$, $\Vdash_{\mathscr{C},\mathfrak{R}_A} \psi$. By our construction of $\mathfrak{R}_a$ we know for every such $\mathscr{C}$ there is a $v\in W$ s.t. $\mathscr{C}=\mathscr{B}_v$ and $R_a wv$. By the inductive hypothesis $\Vdash_{\mathscr{C},\mathfrak{R}_A} \psi$ iff $M,v\Vdash \psi$ and we are done.

With this we can reach our result by step 
 (vi). From $M',w\nvDash \phi$ and (v), it follows that $\nVdash_{\mathscr{B}_w, \mathfrak{R}_A} \phi$. Since $\mathfrak{R}_A$ is a set of $S5$-modal relations, $\phi$ is not valid in the base-extension semantics for $S5$.

We conclude that (for arbitrary $\phi$), if $\phi$ is not valid in the Kripke semantics for $S5$, then it is not valid in the base-extension semantics for epistemic logic either.
    
\end{proof}

\end{document}